\documentclass[11pt]{amsart}
\usepackage{amsthm,amsmath}
\usepackage{amssymb}
\usepackage{braket}
\usepackage{enumerate}
\usepackage{color}
\usepackage[english]{babel}
\usepackage{amsaddr}

\usepackage[a4paper,top=3cm,bottom=4cm]{geometry} 

\newcommand{\be}{\beta}
\newcommand{\de}{\delta}
\newcommand{\e}{\varepsilon}

\newcommand{\CC}{{\mathbb{C}}}

\newcommand{\HH}{{\mathbb{H}}}
\newcommand{\FF}{{\mathbb{F}}}
\newcommand{\PP}{{\mathbb{P}}}

\newcommand{\RR}{{\mathbb{R}}}
\newcommand{\ZZ}{{\mathbb{Z}}}

\newcommand{\NN}{{\mathbb{N}}}

\newcommand{\op}{\operatorname}

\newcommand{\Sp}{\op{Sp}}

\def\diag{\operatorname{diag}}
\def\tr{{\rm Tr}}

\def\det{{\rm det}}

\theoremstyle{plain}
\newtheorem{thm}{Theorem}
\newtheorem*{thm*}{Theorem}

\newtheorem{prop}[thm]{Proposition}
\newtheorem{cor}[thm]{Corollary}
\newtheorem*{cor*}{Corollary}

\theoremstyle{definition}

\newcommand\tc[2]{\vartheta\left[\begin{smallmatrix}#1\\ #2\end{smallmatrix}\right]}
\newcommand{\sm}[2]{\left(\begin{smallmatrix}#1\\#2\end{smallmatrix}\right)}
\newcommand{\ch}[2]{\left[\begin{smallmatrix}#1\\#2\end{smallmatrix}\right]}
\begin{document}

\title{On isomorphisms between Siegel modular threefolds}

\author{Sara Perna}
\address{\footnotesize{Universit\`a degli Studi di Roma ``La Sapienza'', \\Dipartimento di Matematica ``Guido Castelnuovo'', \\Piazzale A. Moro 5, 00185, Roma, Italia.}}
\email{perna@mat.uniroma1.it}

\begin{abstract}
The present paper is motivated by Mukai's ``\emph{Igusa quartic and Steiner surfaces}''. There he proved that the Satake compactification of the moduli space of principally polarized abelian surfaces with a level two structure has a degree 8 endomorphism. This compactification can be viewed as a Siegel modular threefold. The aim of this paper is to show that this result can be extended to other modular threefolds.

The main tools are Siegel modular forms and Satake compactifications of arithmetic quotients of the Siegel upper-half space. Indeed, the construction of the degree 8 endomorphism on suitable modular threefolds is done via an isomorphism of graded rings of modular forms. 

By studying the action of the Fricke involution one gets a further extension of the previous result to other modular threefolds. 

The possibility of a similar situation in higher dimensions is discussed at the end of the paper.

\end{abstract}
\maketitle

\section*{Introduction}

The Satake compactification of the moduli space of principally polarized abelian varieties with level 2 structure is known to be a quartic hypersurface in $\PP^4$. It is called the Igusa quartic. Recently this space has been characterized as a Steiner hyperquartic, as such it has a degree 8 endomorphism. By means of this characterization the Satake compactification  of the  moduli space of  principally polarized abelian  varieties  with  G\"opel triples is isomorphic  to the Igusa quartic~\cite{mukai}. These compactifications can be thought as Siegel modular threefolds and the latter isomorphism can be given by means of Siegel modular forms (cf.~\cite[section 11]{cvg}). 

The aim of this paper is to show that this result can be extended to other modular threefolds. The construction of a degree 8 endomorphism on suitable Siegel modular threefolds will be done via an isomorphism of graded rings of Siegel modular forms and a degree 8 map between two given Siegel modular threefolds. 

Set 
$\HH_g:=\Set{\tau\in\op{M}_{g\times g}(\CC)\mid {}^t\tau=\tau,\;\op{Im}\tau>0}.$
Since $\HH_1$ is the upper-half plane of complex numbers with positive imaginary part, $\HH_g$ is usually called the Siegel upper-half space of degree, or genus, $g$. It is a complex manifold of dimension $g(g+1)/2$. 

As a generalization of M\"obius transformation on the complex plane, preserving the upper-half plane, the group of integral symplectic matices $\Gamma_g$ acts on $\HH_g$ as follows:
\begin{equation}\label{symplectic action}
\gamma\cdot\tau=(A\tau+B)(C\tau+D)^{-1},\;\forall\gamma=\sm{A&B\\}{C&D}\in\Gamma_g,
\end{equation}
where $A,\,B,\,C,\,D$ are $g\times g$ matrices satisfying suitable equations. This block notation for an element of the integral symplectic group will be used throughout the paper.

Let $\Gamma$ be a subgroup of finite index in $\Gamma_g$. It still acts on $\HH_g$ by the same formula~\eqref{symplectic action}. The orbit space $\HH_g/\Gamma$ is called an arithmetic quotient of the Siegel upper-half space. It is smooth if $\Gamma$ acts freely, and it has at most finite quotient singularities in any case. Arithmetic quotients of $\HH_g$ arise naturally in the study of principally polarized abelian varieties over $\CC$ of a given dimension. For example the arithmetic quotient of $\HH_g$ with respect to the full integral symplectic group is known to be the moduli space of principally polarized abelian varieties of dimension $g$.

A singular compactification of these arithmetic quotients is done by means of Siegel modular forms (cf. section~\ref{sec: modular forms}), whose space can be given the structure of a graded ring. Roughly speaking Siegel modular forms are holomorphic functions on $\HH_g$ that behave nicely under the action of the integral symplectic group or its subgroups. This kind of compactification is known as Satake compactification. 

Let us outline the structure of the paper.
First we will consider the rings of Siegel modular forms (with multiplier) with respect to the following two congruence subgroups of level 4 of the integral symplectic group:
\begin{equation}\label{gamma24}
\Gamma_2(2,4)=\Set{\gamma\in\Gamma_2\,\mid\,\gamma\equiv1_4\text{ (mod 2)},\, \diag(B)\equiv\diag(C)\equiv0\text{ (mod 4)}},
\end{equation}
and
\begin{equation}\label{gamma224}
\Gamma_2^2(2,4)=\{\gamma\in\Gamma_2\,|\,
\gamma\equiv\sm{1_2&*}{0&1_2} \text{ (mod 2)}
,\,\diag(2B)\equiv\diag(C)\equiv0\text{ (mod 4)}\},
\end{equation}
where $1_n$ is the $n\times n$ identity matrix and $\op{diag}(M)$ denotes the diagonal of the \mbox{matrix $M$.}  
The computation of such rings will led to the interpretation of the degree 8 map
\begin{equation}\label{square}
\begin{split}
\PP^3&\to\PP^3\\
[x_0, x_1, x_2, x_3] & \mapsto  [x_0^2, x_1^2, x_2^2, x_3^2]
\end{split}
\end{equation} 
as a map between the modular varieties related to these two groups:
\begin{equation}\label{P_3 squared}
\psi:\op{Proj}( A(\Gamma_2(2,4))) \to  \op{Proj}( A(\Gamma_2^2(2,4))).
\end{equation}

This is the part that will give the right degree of the desired self map between some given modular threefolds. The other part, that is the isomorphism of suitable rings of modular forms, involves two more groups, namely the level 2 subgroups
\begin{align}\label{gamma0}
\Gamma_0(2)&=\Set{\gamma\in\Gamma_2\,|\, C\equiv0\text{ (mod 2)}},\\ \label{gamma00}
\Gamma_0^0(2)&=\Set{\gamma\in\Gamma_2\,|\,C\equiv B\equiv0\text{ (mod 2)}}.
\end{align}
We shall prove that there is indeed an isomorphism
\[\Gamma_0(2)/\Gamma_2^2(2,4)\cong\Gamma_0^0(2)/\Gamma_2(2,4),\]
equivariant with respect to the action of the groups on the two copies of $\PP^3$ in~\eqref{P_3 squared}. 
Denoting this group by $G$ we'll establish the following theorem.
\begin{thm*}  For any subgroup $H\subset G$ there exists an isomorphism of graded rings of modular forms
\[\Phi_{H}:A(\Gamma)\to A(\Gamma'),\]
where $
\Gamma_2(2,4)\subset\Gamma\subset\Gamma_0^0(2)$, 
$\Gamma_2^2(2,4)\subset\Gamma'\subset\Gamma_0(2)$
and the quotients $\Gamma/\Gamma_2(2,4) $ and $\Gamma'/\Gamma_2^2(2,4)$ are both isomorphic to $H$.
\end{thm*} 

With a suitable choice of the subgroup $H$ one can actually recover Mukai's result. Indeed by the above theorem and the modular interpretation of the map~\eqref{square} the following corollary follows.
\begin{cor*}
For any subgroup $\Gamma_2(2,4)\subset\Gamma'\subset\Gamma_0^0(2)$  the modular variety $\op{Proj}(A(\Gamma'))$ has a map of degree 8 onto itself.
\end{cor*}


The previous corollary turns out to be true also for other modular threefolds. This will be achieved by  means of the action of the Fricke involution 
\[J_2=\tfrac{1}{\sqrt{2}}\sm{0&1_2}{-2\,1_2&0}.\]

The last section of the paper is dedicated to the case of genus three, with a view toward the situation in higher dimensions. The arguments exploited in genus two do not generalize directly.

\section{Siegel modular varieties}

In this section we recall the definition and properties of Siegel modular forms and modular varieties in order to fix notations.  For a survey article on Siegel modular varieties we refer to~\cite{HS}.
\subsection{Congruence subgroups of the integral symplectic group}
Set $g\in\NN$, $g>1$. Let
\begin{equation}\label{def_sympl}\Sp(2g,\RR)=\Set{\gamma\in \text{M}_{2g\times 2g}(\RR)\,\mid\,{}^t\gamma J\gamma=J},\;\;\text{where}\;\;
J=\sm{0&1_g}{-1_g&0}
,\end{equation}
be the symplectic group of degree $g$. The subgroup
\[\Gamma_g:=\Sp(2g,\ZZ)=\Sp(2g,\RR)\cap \text{M}_{2g\times 2g}(\ZZ)\]
is called the integral symplectic group of degree $g$. We will use a standard block notation for the elements of the integral symplectic group, that is
\[\gamma=\sm{A&B}{C&D}\in\Gamma_g,\] 
where $A,\,B,\,C,\,D\in\text{M}_{g\times g}(\ZZ)$. With these notations it is easily checked that the blocks of a symplectic matrix satisfy the relations
\begin{equation}\label{def symplectic}
\begin{cases}
A\,{}^tB=B\,{}^tA,\\
C\,{}^tD=D\,{}^tC,\\
A\,{}^tD-B\,{}^tC=1_g
\end{cases};\;\;
\begin{cases}
{}^tA\,C={}^tC\,A,\\
{}^tB\,D={}^tD\,B,\\
{}^tA\,D-{}^tC\,B=1_g
\end{cases}.
\end{equation}

In the theory of Siegel modular forms the following subgroups of $\Gamma_g$ are quite remarkable. The principal congruence subgroup of level $n$ of $\Gamma_g$ (the kernel of reduction modulo $n$) is denoted by $\Gamma_g(n)$.
A subgroup $\Gamma\subset\Gamma_g$ such that $\Gamma_g(n)\subset\Gamma$ for some $n\in\NN$ is called a congruence subgroup of level $n$. 
Examples of level $2n$ subgroups are given by the groups
\[\Gamma_g(n,2n)=\{\gamma\in\Gamma_g(n)\;|\;
\mathrm{diag}(\,{}^tA\,C)\equiv\mathrm{diag}(\,{}^tB\,D)\equiv0\;\text{(mod\,} 2n \text{)}\}.\]


\noindent For even values of the level there is a simpler description of such subgroups. It is easily seen that for $\gamma\in\Gamma_g(2m)$ 
\[\begin{cases}
\diag(\,{}^tA\,C)\equiv\diag(C)\;\text{(mod\,} 4m \text{)}
\\
\diag(\,{}^tB\,D)\equiv\diag(B)\;\text{(mod\,} 4m \text{)}\end{cases},\]
hence 
$\Gamma_g(2m,4m)=\{\gamma\in\Gamma_g(2m)\;|\;
\mathrm{diag}(B)\equiv\mathrm{diag}(C)\equiv0\;\text{(mod\,} 4m \text{)}\}.$
Furthermore, for any $m\in\NN$, $\Gamma_g(2m,4m)$ is a normal subgroup of $\Gamma_g$. 

\subsection{Siegel modular forms with multiplier and modular varieties}\label{sec: modular forms}
Throughout the paper we will study rings of Siegel modular forms with multiplier with respect to some congruence subgroups of $\Gamma_g$. 
For a positive integer $k$, a multiplier system of weight $k/2$ for a congruence subgroup $\Gamma$ is a map $v\colon\Gamma\to\CC^\ast$ such that if we define
\begin{align*}
\varphi:\Gamma\times\HH_g\to&\;\CC^*\\
(\gamma,\tau) \mapsto&\;\varphi(\gamma,\tau):=v(\gamma)\det(C\tau+D)^{k/2},
\end{align*}
then it satisfies the cocycle condition
\[\varphi(\gamma\be,\tau)=\varphi(\gamma,\be\cdot\tau)\varphi(\be,\tau)\]
for all $\gamma,\,\beta\in\Gamma$ and $\tau\in\HH_g$. A multiplier system of integral weight is a character for $\Gamma$.
A Siegel modular form of weight $k$/2 with respect to $\Gamma$ and the multiplier system $v$ is a holomorphic function $f\colon\HH_g\to\CC$ such that
\begin{equation*}
f(\gamma\cdot\tau)=v(\gamma)\det(C\tau+D)^{k/2} f(\tau),\quad \forall\gamma\in\Gamma.
\end{equation*}
Equivalently,
\begin{equation*}
f_{|\gamma,k/2,v}(\tau)=f(\tau),\;\forall\gamma\in\Gamma,\,\forall\tau\in\HH_g,
\end{equation*}
where
\begin{equation}\label{action2}
f_{|\gamma,k/2,v}(\tau)=v(\gamma)^{-1}\det(C\tau+D)^{-k/2}f(\gamma\cdot\tau).
\end{equation}
Sometimes, to ease notation, the multiplier system and the weight will be omitted in~\eqref{action2}. The set $[\Gamma,k/2,v]$ of such functions is known to be a finite dimensional vector space. If $v$ is a multiplier system of weight $1/2$, the ring of Siegel modular forms with respect to the group $\Gamma$ and the multiplier system $v$ is the graded ring
\[A(\Gamma,v)=\bigoplus_{k\in\NN}[\Gamma,k/2,v^k].\]

In addition, $A(\Gamma,v)$ is known to be a normal integral domain of finite type over $\CC$. 
The complex projective variety $\op{Proj}(A(\Gamma,v))$ is called the modular variety associated to $\Gamma$. It contains the arithmetic quotient $\HH_g/\Gamma$ as a Zariski open set, so it is natural to define its compactification as
\[\overline{\HH_g/\Gamma}=\op{Proj}(A(\Gamma,v)).\]
This is also known as the Satake compactification of the arithmetic quotient $\HH_g/\Gamma$. This compactification does not depend on the multiplier system chosen and if we let
\[A(\Gamma,v)^{(d)}=\bigoplus_{k\equiv0\,(\text{mod}\,d)}[\Gamma,k/2,v^k],\]
then $\op{Proj}(A(\Gamma,v))\cong \op{Proj}(A(\Gamma,v)^{(d)})$.

\subsection{Theta constants}\label{sec: theta con}
Theta constants with characteristics are classical examples of modular forms with multiplier if we consider their transformation under the action of suitable congruence subgroups of the integral symplectic group. 
Given a theta characteristic, that is a column vector $m=\ch {m'}{m''}\in\FF_2^{2g}$, the theta function with characteristic $\vartheta_m:\HH_g\times\CC^g\to\CC$ is defined by the series
\begin{equation*}
\vartheta_m(\tau,z)=
\sum_{n\in\ZZ^g}\mathtt{e}\left(\frac{1}{2} {}^t(n+m'/2)\tau(n+m'/2)+{}^t(n+m'/2)(z+m''/2)\right),
\end{equation*}
where $\mathtt{e}(\cdot)=e^{2\pi i(\cdot)}.$ It is a classical result that such a series defines an holomorphic function on $\HH_g\times\CC^g$ (cf.~\cite{rafabook}).
The theta function is an even (odd) function of $z$ if ${}^tm'm''\equiv0\pmod{2}$ (${}^tm'm''\equiv1\pmod{2}$), and so the parity of a theta characteristic is defined as
\[e(m)=(-1)^{{}^tm'm''},\]
and a theta characteristic $\ch {m'}{m''}$ is called even if $e(m)=1$ and odd if $e(m)=-1$.  There are $2^{g-1}(2^g+1)$ even characteristics and $2^{g-1}(2^g-1)$ odd ones. Evaluating the theta function at $z=0$ we get a holomorphic function on the Siegel upper-half space, not identically zero if and only if the related theta characteristic is even. These functions are usually called theta constants and are denoted by
\[\tc{m'}{m''}(\tau)=\tc{m'}{m''}(\tau,0).\]
When convenient, we will denote by $\vartheta_m$ the theta constant with characteristic $m=\ch {m'}{m''}$.

The integral symplectic group acts not only on the Siegel upper-half space by~\eqref{symplectic action} but also on the set of theta characteristics. We will see that via these actions the integral symplectic group acts on theta constants. For $\gamma\in\Gamma_g$ the action on theta characteristics is defined as follows. Let 
\begin{equation}\label{action characteristics}
\gamma\cdot\begin{bmatrix}
m'\\ m''\end{bmatrix}=\left[
\begin{pmatrix}
D&-C\\-B&A\end{pmatrix}\begin{pmatrix}
m'\\m''\end{pmatrix}+
\begin{pmatrix}
\diag(C^tD)\\\diag(A^tB)\end{pmatrix}
\right]\pmod{2}.
\end{equation}
The action defined in this way is neither linear nor transitive. Indeed, the action preserves the parity of the characteristics.

An important property of theta constants is the  classical transformation formula (cf~\cite{ig1}), namely theta constants with characteristics satisfy the identity
\begin{equation}\label{trans_theta}
\vartheta_{\gamma\cdot m}(\gamma\cdot\tau)=\kappa(\gamma)\det(C\tau+D)^{1/2}
\texttt{e}(\varphi_m(\gamma))\vartheta_m(\tau),
\end{equation}
where $\kappa(\gamma)$ is a primitive 8\textsuperscript{th} root of unity depending on $\gamma$
and
\begin{equation}\label{phi_m}
\varphi_m(\gamma)=-\frac{1}{8}({}^tm'{}^tBDm'+{}^tm''{}^tACm''-2\,{}^tm'{}^tBCm'')+\frac{1}{4}\,{}^t\op{diag}(A{}^tB)(Dm'-Cm'').
\end{equation}
An explicit expression is known for suitable powers of $\kappa$. For example, 
\begin{align}\label{kappa4}
\kappa(\gamma)^4&=(-1)^{\op{Tr}({}^tBC)},\quad\;\;\, \forall\gamma\in\Gamma_g,\\\label{kappa2}
\kappa(\gamma)^2&=(-1)^{\tr\left(\frac{A-I_g}{2}\right)}, \quad \forall\gamma\in\Gamma_g(2).
\end{align}

\noindent The transformation rule for theta constants under the action of an element $\gamma\in\Gamma_g$ can be easily derived from~\eqref{trans_theta}. That is
\[\vartheta_m(\gamma\cdot\tau)=\kappa(\gamma)\det(C\tau+D)^{1/2}\texttt{e}(\varphi_{\,\gamma^{-1}\cdot m}(\gamma))\vartheta_{\,\gamma^{-1}\cdot m}(\tau).\]
If $\gamma\in\Gamma_g(4,8)$ it is easy to see that $\gamma^{-1}\cdot m=m$ and $\texttt{e}(\varphi_{\,\gamma^{-1}\cdot m}(\gamma))=1$. Thus, theta constants with characteristic are modular forms of weight $1/2$ with respect to $\Gamma_g(4,8)$ and the multiplier system $v_\vartheta(\gamma):=\kappa(\gamma)$.

The modular forms we will be mainly interested in are the second order theta constants. For any $a\in\FF_2^g$, define
\[\Theta[a](\tau)=\tc {a}0(2\tau).\]
By the above formula we can immediately deduce the transformation rule for these functions under the action of an element $\gamma\in\Gamma_g$. For any $\gamma\in\Gamma_g$, let $\tilde{\gamma}\in\Gamma_g$ be the matrix such that $2(\gamma\cdot\tau)=\tilde{\gamma}\cdot2\tau$, namely $\tilde{\gamma}=\sm{A&2B}{C/2&D}$. Note that $\tilde{\gamma}\in\Gamma_g$ if and only if $C\equiv0\pmod{2}$. Therefore we get the following transformation rule:
\[\Theta[a](\gamma\cdot\tau)=\kappa(\tilde{\gamma})\det(C\tau+D)^{1/2}\texttt{e}(\varphi_{\,\tilde{\gamma}^{-1}\cdot \ch a0}(\tilde{\gamma}))\Theta[a'](\tau), \quad \gamma\in\Gamma_{g,0}(2),\]
where $a'=(\tilde{\gamma}^{-1}\cdot \ch a0)'$ and $\Gamma_{g,0}(2)=\{\gamma\in\Gamma_g\,|\,C\equiv0\,\text{(mod\,} 2 \text{)}\}$.
Then it is easily seen that second order theta constants are modular forms of weight $1/2$ with respect to the group $\Gamma_g(2,4)$ and the multiplier system $v_\Theta(\gamma)=\kappa(\tilde{\gamma})$.

Riemann's addition formula relates theta constants and second order theta constants(cf. \cite[Appendix II to Chapter II]{rafabook}).
\begin{prop}[Riemann's addition formula]
Given two $g$-characteristics $\ch{\e}{\e'}$, $\ch{\de}{\de'}$, for any $z,\,w\in\CC^g$ and $\tau\in\HH_g$ we have
\[\vartheta\bigl[\begin{smallmatrix}\varepsilon \\\varepsilon'\end{smallmatrix}\bigr]
\left(\begin{smallmatrix}\tau, &\frac{z+w}{2}\end{smallmatrix}\right)
\vartheta\bigl[\begin{smallmatrix}\delta \\\delta'\end{smallmatrix}\bigr]
\left(\begin{smallmatrix}\tau, &\frac{z-w}{2} \end{smallmatrix}\right)
=
\sum_{\sigma\in\mathbb{F}_2^g}
\vartheta\Bigl[\begin{smallmatrix}\frac{\varepsilon+\delta}{2}-\sigma\\
\varepsilon'+\delta'\end{smallmatrix}\Bigr]
(2\tau,z)
\vartheta\Bigl[\begin{smallmatrix}
\frac{\varepsilon-\delta}{2}+\sigma\\\varepsilon'-\delta'
\end{smallmatrix}\Bigr]
(2\tau,w).\]
\end{prop}

\noindent By the above proposition the following well known identities can be derived:
\begin{align}\label{riemann relation}
\Theta[\sigma](\tau)\Theta[\sigma+\varepsilon](\tau)&=\frac{1}{2^g}\sum_{\varepsilon'\in\mathbb{F}_2^g}
(-1)^{\sigma^t\varepsilon'}\vartheta^2\bigl[\begin{smallmatrix}\varepsilon \\\varepsilon'\end{smallmatrix}\bigr](\tau),\\[6pt]
\label{riemann relation m1}
\vartheta^2\bigl[\begin{smallmatrix}\varepsilon \\\varepsilon'\end{smallmatrix}\bigr](\tau)
&=
\sum_{\sigma\in\mathbb{F}_2^g}(-1)^{\sigma^t\varepsilon'}
\Theta[\sigma](\tau)\Theta[\sigma+\varepsilon](\tau).
\end{align}
A first consequence of these relations is that $\kappa(\tilde{\gamma})^2=\kappa(\gamma)^2$. Hence the square of $v_\Theta$ is a non-trivial multiplier system 
on $\Gamma_g(2,4)$.
Moreover, it also follows that the vector space of modular forms spanned  by $\{\Theta[a]^2\}_{a \in \FF_2^g}$ coincides with the one spanned  by $\{\tc 0b^2\}_{b\in\FF_2^g}$. From now on we set the notation $\vartheta_b:=\tc 0b$.

\subsection{Differential forms and vector-valued modular forms}\label{diff_form}

Let $(\rho,V_\rho)$ be an irreducible rational representation of $\op{GL}(g,\CC)$ on a finite dimensional complex vector space $V_\rho$. Let $\Gamma$ be a congruence subgroup of $\Gamma_g$. A holomorphic function $f:\HH_g\to V_\rho$ is called a vector valued modular form with respect to $\Gamma$ and the representation $\rho$ if
\[f_{\mid_{\gamma,\rho}}(\tau):=\rho(C\tau+D)^{-1}f(\gamma\cdot\tau)=f(\tau),\;\forall\gamma\in\Gamma.\]
The space of such functions will be denoted by $[\Gamma,\rho]$. 
The representation $\rho$ is uniquely identified by its highest weight $(\lambda_1,\dots,\lambda_g)\in\ZZ^g$ with $\lambda_1\ge\cdots\ge\lambda_g$. The weight of the vector valued modular form is defined to be equal \mbox{to $\lambda_g$.} Siegel modular forms with trivial multiplier of section~\ref{sec: modular forms} are vector valued modular forms with respect to the representation $\det^k$, for a suitable $k$.

For any complex manifold $X$ denote by $\Omega^n(X)$ the sections of the sheaf of holomorphic differential forms on $X$ of degree $n$. For a congruence subgroup $\Gamma$ denote by $X_\Gamma^0$ the set of regular points of $\overline{\HH_g/\Gamma}$ and by $\widetilde{X_\Gamma}$ a desingularization of $\overline{\HH_g/\Gamma}$. If $N$ is the dimension of $\HH_g$, by~\cite{FP} every holomorphic differential form $\omega\in\Omega^n(X_\Gamma^0)$ of degree $n<N$ extends holomorphically to $\widetilde{X_\Gamma}$. Moreover, if $g\ge2$ and $n<N$ there is a natural isomorphism 
\[\Omega^n(X_\Gamma^0)\cong\Omega^n(\HH_g)^\Gamma,\]
where $\Omega^n(\HH_g)^\Gamma$ is the space of $\Gamma$-invariant holomorphic differential forms on $\HH_g$ of degree $n$.
For suitable degrees, depending only on $g$, some of these spaces are known to be trivial. The non-trivial $\Gamma$-invariant holomorphic differential forms have a precise description in terms of vector-valued modular forms (cf~\cite{weissauer}).   

In this paper we will be mainly interested in holomorphic differential forms of degree $N-1$. 
On $\HH_g$, define 
\[d\check{\tau}_{ij}=\pm\; e_{ij}\bigwedge_{\substack{1\le k\le l\le g\\[2pt](k,l)\neq(i,j)}}d\tau_{kl};\;\; e_{ij}=\frac{1+\de_{ij}}{2},\]
where the sign is so chosen that  $d\check{\tau}_{ij}\wedge d\tau_{ij}=e_{ij}\bigwedge_{1\le k\le l\le g}d\tau_{kl}$. Then any $\omega\in\Omega^{N-1}(\HH_g)^\Gamma$ can be written in the form 
\begin{equation}\label{omega}
\omega=\op{Tr}(f(\tau)d\check{\tau})=\sum_{1\le i,j\le g} f_{ij}(\tau)d\check{\tau}_{ij},
\end{equation}
where
\begin{equation*}
f(\gamma\cdot\tau)=\op{det}(C\tau+D)^{g+1}\,{}^t(C\tau+D)^{-1}f(\tau)(C\tau+D)^{-1}.
\end{equation*}
Equivalently, $f$ is a vector valued modular form with respect to $\Gamma$ and the irreducible representation $\rho_1$ with highest weight $(g+1,\dots,g+1,g-1)$.

\section{Isomorphic Siegel modular threefolds with a degree 8 endomorphism}\label{section2}

In this section we will focus our attention on the case $g=2$. We will show that many modular threefolds share with the Igusa quartic the property of having a degree 8 endomorphism. To be precise, this is the result of Corollary~\ref{gen1} and Corollary~\ref{gen2}. For this purpose we will construct isomorphism of graded rings of Siegel modular forms giving isomorphisms of the related modular varieties.

We will be mostly interested in the four second order theta constants in genus 2.
\[f_{00}:=\Theta[0\,0],\;f_{01}:=\Theta[0\,1],\;f_{10}:=\Theta[1\,0],\;f_{11}:=\Theta[1\,1].\]
It is shown in~\cite{runge1} that
\begin{equation}\label{ring Gamma(2,4)}
A(\Gamma_2(2,4),\,v_\Theta)=\CC[f_{00},f_{01},f_{10},f_{11}].
\end{equation}
Hence the modular threefold associated to this ring is  isomorphic to $\PP^3$.
 
\subsection{Degree 8 map between two different modular threefolds}

Any symmetric $2\times 2$ integer matrix $S$ determines an element $\gamma_S\in\Gamma_2$, namely
\[\gamma_S=\begin{pmatrix}1_2&S\\0&1_2\end{pmatrix}.\]
In particular, if we put
\[B_1=\sm{2&0}{0&0},\;B_2=\sm{0&0}{0&2},\;B_3=\sm{0&1}{1&0},\]
then the matrices $M_i:=\gamma_{B_i}$ belong to $\Gamma_2^2(2,4)$, and the $M_i^2$ belong to its index 8 normal subgroup $\Gamma_2(2,4)$. By taking $\{M_1,M_2,M_3\}$ as a basis we thus identify $\Gamma_2^2(2,4)/\Gamma_2(2,4)$ with $\FF_2^3$.

We will discuss the action of this group on the second order theta constants $f_a$ and the functions $\vartheta_b$ with $a,\,b\in\FF_2^2$, in order to find generators for the ring of modular forms with respect to $\Gamma_2^2(2,4)$ and a suitable multiplier system. First focus on the action of the matrices $M_i$ on theta constants. From~\cite[p. 59]{runge1} we have
\[\tc {m'}{m''}(\gamma_S\cdot\tau)=\tc {m'}{m''}(\tau+S)=\varepsilon^{-{}^tm'(Sm'+2\diag(S))}\tc {m'}{m''+Sm'+\diag(S)}(\tau),\]
with $\varepsilon=\frac{1+i}{\sqrt{2}}$ a primitive 8\textsuperscript{th} root of unity. For the second order theta constants this gives
\[\Theta[a](\gamma_S\cdot\tau)=i^{{}^taSa}\,\Theta[a](\tau).\]
Thus, for $a=(a_1,a_2)\in\FF_2^2$ it follows that
\begin{align*}
f_a(M_1\cdot\tau)&=(-1)^{a_1}f_a(\tau),\\
f_a(M_2\cdot\tau)&=(-1)^{a_2}f_a(\tau),\\
f_a(M_3\cdot\tau)&=(-1)^{a_1a_2}f_a(\tau).
%
\end{align*}
So the group $\Gamma_2^2(2,4)/\Gamma_2(2,4)$ acts by changes of sign on $f_a$ if $a\neq 00$. Therefore it acts trivially on the $f_a^2$ and also on the $\vartheta_b^2$. 
\begin{prop} 
The ring $\CC[f_{00}^2,\dots,f_{11}^2]$ is equal to the subring $A_\NN(\Gamma_2^2(2,4),\,v_{\Theta}^2)\subset A(\Gamma_2^2(2,4),\,v_{\Theta}^2)$ of modular forms with integer weight. 
\end{prop}

\begin{proof}
We have just seen that $\CC[f_a^2]\subset A_\NN(\Gamma_2^2(2,4),\,v_{\Theta}^2)$. For the opposite inclusion, since both rings are integrally closed it is enough to show that they have the same field of fractions. This is also immediate, because we have already seen that they both have $\CC(f_a)$ as an extension of degree 8.
 \end{proof}

Thus the degree 8 endomorphism of $\PP^3$ given by 
$[x_0,\dots,x_3]\mapsto[x_0^2,\dots,x_3^2]$
can be seen as a map between the two  modular varieties
$$\psi :\op{Proj}( A(\Gamma_2(2,4))) \to  \op{Proj}(A(\Gamma_2^2(2,4))).$$
Here we omit the multipliers since the modular variety is independent of the choice of the multiplier system.

\subsection{Isomorphisms of modular threefolds, the main result}
It is easily checked that the groups $\Gamma_2^2(2,4)$ and $\Gamma_2(2,4)$ are normal subgroups of $\Gamma_0(2)$ and $\Gamma_0^0(2)$ respectively (see~\eqref{gamma0},~\eqref{gamma00},~\eqref{def symplectic}). 
Moreover $[\Gamma_0(2)\colon\Gamma_2^2(2,4)]=96$ and the quotient group is isomorphic to the semidirect product $\FF_2^4\ltimes S_3$ where $S_3$ is the symmetric group of order 3.

We can construct an isomorphism, 
\[\varphi\colon\Gamma_0(2)/\Gamma_2^2(2,4)\to\Gamma_0^0(2)/\Gamma_2(2,4),\]
as follows. For a class $\gamma\in\Gamma_0(2)/\Gamma_2^2(2,4)$ we can choose a representative (which we also call $\gamma$) of the form
\begin{equation*}
\gamma\equiv\sm{A&B}{C&D}=\sm{1&0}{CA^{-1}&1}\sm{A&0}{0&^tA^{-1}}\sm{1&A^{-1}B}{0&1}.
\end{equation*}
Define \[\varphi(\gamma)=\sm{1&0}{CA^{-1}&1}\sm{A&0}{0&^tA^{-1}}\sm{1&A^{-1}2B}{0&1}.\]
Roughly speaking, the map $\varphi$ sends ``B'' to ``2B''. Set \[G:=\Gamma_0(2)/\Gamma_2^2(2,4)\cong\Gamma_0^0(2)/\Gamma_2(2,4).\]
If $S$ is a symmetric $2\times 2$ matrix with integer coefficients, then from~\cite[section 2]{fsm} we know that $\Gamma_0(2)$ is generated by matrices of the form ${}^t\gamma_{2S}$, $\gamma'=\sm{A&0}{0&^tA^{-1}}$ and $\gamma_S$. The classes of these matrices are then generators for the group $\Gamma_0(2)/\Gamma_2^2(2,4)$ and their images under $\varphi$ are generators for the group $\Gamma_0^0(2)/\Gamma_2(2,4)$.

Omitting the weight and the multiplier in the notation of~\eqref{action2}, by an easy computation it follows that
\[
\begin{array}{ll}
{f_a}^2_{\mid_{{}^t\gamma_{2S}}}=f_{a-\diag(S)}^2,&
{f_a}_{\mid_{{}^t\gamma_{2S}}}=f_{a-\diag(S)}.\\
\\
{f_a}^2_{\mid_{\gamma'}}=f_{Aa}^2,&{f_a}_{\mid_{\gamma'}}=f_{Aa},\\
\\
{f_a}^2_{\mid_{\gamma_S}}=i^{{}^ta2Sa}f_a^2
,&{f_a}_{\mid_{\gamma_{2S}}}=i^{{}^ta2Sa}f_a.
\end{array}
\]
This shows that via the isomorphism $\varphi$ the action of the group $G$ on the two polynomial rings is the same and the map $f_a\mapsto f_a^2$ is an isomorphism of $G$-modules. 
\begin{thm}\label{prop gamma02}  For any subgroup $H\subset G$ there exist two groups $\Gamma,\,\Gamma'$ such that 
\[
\Gamma_2(2,4)\subset\Gamma\subset\Gamma_0^0(2),\;\;
\Gamma_2^2(2,4)\subset\Gamma'\subset\Gamma_0(2)
\]
and the quotients $\Gamma/\Gamma_2(2,4) $ and $\Gamma'/\Gamma_2^2(2,4)$ are both isomorphic to $H$ via the map induced by $\varphi$. It is also induced an isomorphism 
\[\Phi_{H}:A(\Gamma,v_\Theta)\to A_\NN(\Gamma',v_\Theta^2),\]
which doubles the weights. If $f\in[\Gamma,k/2,v_\Theta]$ then $\Phi_H(f)\in[\Gamma',k,v_\Theta^2]$.
\end{thm} 
\noindent As an immediate consequence we have the following corollary:
\begin{cor}\label{gen1}
For every subgroup $\Gamma$ such that $\Gamma_2(2,4)\subset\Gamma\subset\Gamma_0^0(2)$ the projective variety $\op{Proj}(A(\Gamma))$ has a map of degree 8 onto itself.
\end{cor} 

\subsubsection{Degree 8 endomorphism of the Igusa quartic}
$\HH_2/\Gamma_2(2)$ is the coarse moduli space of  principally polarized abelian surfaces with a level 2 structure. Igusa~\cite{igusa} proved that the Satake compactification $\overline{\HH_2/\Gamma_2(2)}$ is a quartic hypersurface in $\PP^4$ given by the equation
\[(x_0x_1+x_0x_2+x_1x_2-x_3^2)^2-4x_0x_1x_2(x_0+x_1+x_2+x_3+x_4)=0.\]

Since $\Gamma_2(2,4)\subset\Gamma_2(2)\subset\Gamma_0^0(2)$, one recovers the result in~\cite{mukai} that the Igusa quartic has a degree 8 endomorphism as a special case of Corollary~\ref{gen1}.

\subsubsection{An isomorphism between the Satake compactifications of two moduli spaces of abelian surfaces}
As another application of Theorem~\ref{prop gamma02} we will give a different proof of the isomorphism between the Igusa quartic and the Satake compactification of the moduli space of principally polarized abelian surfaces with a G\"opel structure studied in~\cite{mukai,cvg}. This moduli space is $\HH_2/\Gamma_1(2)$, where
\[\Gamma_1(2)=\Set{\gamma\in\Gamma_2\mid A\equiv D\equiv 1_2\pmod{2},\,C\equiv0\pmod{2}}.\]
It is readily seen that both $\Gamma_2(2)/\Gamma_2(2,4)$ and 
$\Gamma_1(2)/\Gamma_2^2(2,4)$ are isomorphic to the group $H$ generated by $M_1,\,M_2,\,{}^tM_1$ and ${}^tM_2$. Therefore the isomorphism $\varphi_H$ of Theorem~\ref{prop gamma02} induces an isomorphism between $\overline{\HH_2/\Gamma_2(2)}$ and $\overline{\HH_2/\Gamma_1(2)}$. 

\subsection{Action of the Fricke involution}

The Fricke involution on $\HH_g$ is the involution given by the matrix
\[J_2=\tfrac{1}{\sqrt{2}}\begin{pmatrix}
0&1_g\\
-2\,1_g&0
\end{pmatrix}\in\Sp(4,\RR).\]
The action of the real symplectic group on $\HH_g$ is defined as in~\eqref{symplectic action}, then for $\tau\in\HH_g$ we have that
\[J_2\cdot\tau=-\frac{1}{2\,\tau}.\]
We are interested in the case $g=2$ and the action of $J_2$ on the functions $f_a$ with  $a\in\FF_2^2$.
Although formula~\eqref{action characteristics} does not define an action of $\Sp(2g,\RR)$ on theta characteristics, it is still possible to use the classical transformation formula for theta functions to compute the action of the matrix $J_2$ on theta constants. 

An easy computation shows that
$f_a(J_2\cdot\tau)=
v_\Theta(J_2)\det(\tau)^{1/2}\vartheta_a(\tau),$ where we define  $v_\Theta(J_2)$ to be equal to $v_\vartheta(J)$ with $J$ as in~\eqref{def_sympl}. Therefore 
\begin{equation}\label{Fricke f_a}
{f_a}_{\mid_{J_2,1/2,v_\Theta}}=\vartheta_a.
\end{equation}
For any $\gamma\in\Sp(4,\RR)$ we write $\gamma^{J_2}$ for the conjugate $J_2\gamma J_2^{-1}$. Then
\begin{equation}\label{conj J2}
\gamma^{J_2}=\begin{pmatrix}D&-C/2\\-2B&A\end{pmatrix},\quad\forall\gamma\in\Gamma_2.
\end{equation}
In particular, if $\gamma\in\Gamma_2$, then $\gamma^{J_2}\in \Gamma_2$ if and only if $C\equiv0\text{\;(mod 2)}$.

From~\eqref{conj J2} we can compute that
\[\Gamma_2^2(2,4)^{J_2}=\Gamma_2^2(2,4)\;\;\text{and}\;\;\Gamma_0(2)^{J_2}=\Gamma_0(2),\] 
whereas
\[\Gamma_0^0(2)^{J_2}=\Gamma_0(4):=\Set{\gamma\in\Gamma_2\,\mid\,C\equiv0\text{ (mod 4)}},\]
and
\[\Gamma_2(2,4)^{J_2}=\left.
\begin{cases}&
A\equiv D\equiv 1_2\text{ (mod 2)},\\
\gamma\in\Gamma_2\;\;\text{s.t.}&C\equiv0\text{ (mod 4)},\;\diag(C)\equiv0\text{ (mod 8)},\\
&\diag(B)\equiv0\text{ (mod 2)}
\end{cases}
\right\}.\]
We can exploit this action to compute the ring of modular forms with respect to the group $\Gamma_2(2,4)^{J_2}$. 
From~\eqref{ring Gamma(2,4)} and~\eqref{Fricke f_a} it follows that
\[A(\Gamma_2(2,4)^{J_2},v_\vartheta)=\CC[\vartheta_b].\]
Moreover, since the $f_a^2$ are linear combination of the $\vartheta_b^2$ and vice-versa, by~\eqref{riemann relation} and~\eqref{riemann relation m1}, the polynomial ring $\CC[f_a^2]=\CC[\vartheta_b^2]$ is invariant under the action of the Fricke involution. 

Thus, we have another modular interpretation of the endomorphism~\eqref{square} of $\PP^3$. Set $G':=\Gamma_0(4)/\Gamma_2(2,4)^{J_2}$. With the same arguments that led us to Theorem~\ref{prop gamma02}, we have an isomorphism $\varphi': \Gamma_0(2)/\Gamma_2^2(2,4)\to\Gamma_0(4)/\Gamma_2(2,4)^{J_2}$ such that via this isomorphism the action of the group $G'$ on the rings $\CC[\vartheta_b]$ and $\CC[\vartheta_b^2]$ is the same and the map $\vartheta_b\mapsto\vartheta_b^2$ is an isomorphism of $G'$-modules.

\begin{thm}\label{prop gamma04}
For any subgroup $H'\subset G'$ there exist two groups $\Delta,\,\Delta'$ such that
\[\Gamma_2(2,4)^{J_2}\subset\Delta\subset\Gamma_0(4), \;\; \Gamma_2^2(2,4)\subset\Delta'\subset\Gamma_0(2),\]
and the quotients $\Delta/\Gamma_2(2,4)^{J_2}$ and $\Delta'/\Gamma_2^2(2,4)$ are both isomorphic to $H'$. This isomorphism is induced by $\varphi'$. Therefore it is also induced an isomorphism of graded ring of modular forms 
\[\Psi_{H'}:A(\Delta,v_\vartheta)\to A_\NN(\Delta',v_\vartheta^2)\]
such that if $f\in[\Delta,k/2,v_\vartheta]$ then $\Psi_{H'}(f)\in[\Delta',k,v_\vartheta^2]$. 
\end{thm}

Note that since the groups $\Gamma_2^2(2,4)$ and $\Gamma_0(2)$ are fixed by the Fricke involution the set of groups between them is also fixed, but the individual groups need not be.
\begin{cor}\label{gen2}
For every subgroup $\Delta$ such that $\Gamma_2(2,4)^{J_2}\subset\Delta\subset\Gamma_0(4)$ the projective variety $\op{Proj}(A(\Delta))$ has a degree 8 endomorphism.
\end{cor} 

\section{Higher dimensions, the genus 3 case}

The arguments of Section~\ref{section2} do not generalize directly to the genus three case. 
The first key point in the genus two case is that there is a map
\[\psi:\overline{\HH_2/\Gamma_2(2,4)} \to  \overline{\HH_2/\Gamma_2^2(2,4)}\]  
which is actually the endomorphism of $\PP^3$ given by 
$
[x_0, x_1, x_2, x_3] \mapsto  [x_0^2, x_1^2, x_2^2, x_3^2]$.

Similar to what we have done in genus two, we define the group 
\[\Gamma_3^2(2,4)=\{\gamma\in\Gamma_3\,|\,
\gamma\equiv\sm{1_2&*}{0&1_2} \text{ (mod 2)}
,\,\diag(2B)\equiv\diag(C)\equiv0\text{ (mod 4)}\}.\]
We will show that both $\overline{\HH_3/\Gamma_3(2,4)}$ and $\overline{\HH_3/\Gamma_3^2(2,4)}$ are not unirational, therefore a map between these two modular varieties is not a map between two projective spaces. A necessary condition for unirationality is that there are no non-trivial holomorphic differential forms in any degree. Exploiting the construction of holomorphic differential forms by means of gradients of odd theta functions introduced in~\cite{sm}, we will show that $\Omega^5(\HH_3)^{\Gamma_3(2,4)}$ and $\Omega^5(\HH_3)^{\Gamma_3^2(2,4)}$ are not trivial. 

Given a matrix $M=(m_1,\dots,m_{g-1})\in M_{2g\times(g-1)}$ of odd theta characteristics with $m_i\neq m_j$ for $1\le i<j\le g-1$ define
\[W(M)(\tau)=\pi^{-2g+2}\;(\psi_{m_1}(\tau)\wedge\cdots\wedge\psi_{m_{g-1}}(\tau)){}^t(\psi_{m_1}(\tau)\wedge\cdots\wedge\psi_{m_{g-1}}(\tau)),\]
where $\psi_{m_i}(\tau)=\op{grad}_z\theta_{m_i}(\tau,0)$, $1\le i\le g-1$. The form $W(M)(\tau)$ does not vanish identically because of the non-vanishing of the Jacobian determinant $\psi_{m_1}(\tau)\wedge\cdots\wedge\psi_{m_g}(\tau)$ whenever we deal with distinct odd characteristics $m_i$ (cf.~\cite{sm}).
For $\gamma\in\Gamma_g$ the following transformation formula holds:
\begin{equation*}
W(\gamma\cdot M)(\gamma\cdot\tau)=\chi_M(\gamma)\op{det}(C\tau+D)^{g+1}\,{}^t(C\tau+D)^{-1}W(M)(\tau)(C\tau+D)^{-1},
\end{equation*}
where 
\[\chi_M(\gamma)=\kappa(\gamma)^{2g-2}\texttt{e}\left(2\textstyle\sum_{i=1}^{g-1}\varphi_{m_i}(\gamma)\right),\] 
with $\texttt{e}(t)=e^{2\pi it}$ and $\varphi_{m_i}(\gamma)$ defined as in~\eqref{phi_m}.
Denoting by $\rho_1$ the irreducible representation of $\op{GL}(g,\CC)$ with highest weight $(g+1,\dots,g+1,g-1)$ (cf. Section~\ref{diff_form}), it is easily seen that $W(M)$ is in $[\Gamma_g(2,4),\rho_1]$ for $g$ odd while it is in $[\Gamma_g^*(2,4),\rho_1]$ for $g$ even, where $\Gamma_g^*(2,4):=\left\lbrace \gamma\in\Gamma_g(2,4)\, |\,\kappa(\gamma)^2=1 \right\rbrace$.

By means of the description of a holomorphic differential form of degree $N-1$ in terms of a suitable vector valued modular form given in formula~\eqref{omega}, it follows in particular that $\Omega^5(\HH_3)^{\Gamma_3(2,4)}$ is non-trivial and so $\overline{\HH_3/\Gamma_3(2,4)}$ is not unirational. Actually in this way one can build at least ${28\choose 2}=378$ non-trivial holomorphic differential forms on $\HH_3$ invariant under the action of $\Gamma_3(2,4)$, each coming from a $W(M)$ where $M$ is a matrix of two distinct odd characteristics.

\begin{thm}
The space $\Omega^5(\HH_3)^{\Gamma_3^2(2,4)}$ is not trivial and so $\overline{\HH_3/\Gamma_3^2(2,4)}$ is not unirational.
\end{thm}

\proof
A way to construct vector valued modular forms that can be used to define holomorphic differential forms on $\HH_3$ invariant under $\Gamma_3^2(2,4)$ is to symmetrize suitable $W(M)$ and check that the resulting vector valued modular form does not vanish identically. 

Given $M=(m_1,m_2)\in M_{6,2}(\FF_2)$, consider
\begin{equation}\label{Phi}
\begin{split}
\Phi(M)(\tau)&=\sum_{\gamma\in\Gamma_3^2(2,4)/\Gamma_3(2,4)}W(M)_{|_{\gamma,\rho_1}}(\tau)\\
&=\sum_{\gamma\in\Gamma_3^2(2,4)/\Gamma_3(2,4)}\kappa(\gamma)^4\,
\texttt{e}\left(2\textstyle\varphi_{n_1}(\gamma)+2\varphi_{n_2}(\gamma)\right)W(N)(\tau),
\end{split}
\end{equation}
where $N=(n_1,n_2)$ with $n_i=\gamma^{-1}\cdot m_i$, $i=1,2$. If well defined and not identically zero, $\Phi(M)$ is a vector valued modular form with respect to $\Gamma_3^2(2,4)$ and the representation $\rho_1$ by construction.

From section~\ref{sec: theta con} we know that $\kappa(\gamma)^4=(-1)^{\op{Tr}({}^tBC)}$ for $\gamma\in\Gamma_g$. It is easily seen that a set of generators for the quotient is given by the classes of the matrices $M_1,\dots,M_6$, where $M_i=\sm{1_3&B_i\\}{0&1_3} $ and
\begin{align*}
B_1=\left(\begin{smallmatrix}
2&0&0\\
0&0&0\\
0&0&0
\end{smallmatrix}\right),\,
B_2=\left(\begin{smallmatrix}
0&0&0\\
0&2&0\\
0&0&0
\end{smallmatrix}\right),\,
B_3=\left(\begin{smallmatrix}
0&0&0\\
0&0&0\\
0&0&2
\end{smallmatrix}\right),\\
B_4=\left(\begin{smallmatrix}
0&1&0\\
1&0&0\\
0&0&0
\end{smallmatrix}\right),\,
B_5=\left(\begin{smallmatrix}
0&0&1\\
0&0&0\\
1&0&0
\end{smallmatrix}\right),\,
B_6=\left(\begin{smallmatrix}
0&0&0\\
0&0&1\\
0&1&0
\end{smallmatrix}\right).
\end{align*} 
Thus, the sum in~\eqref{Phi} is finite and $\Phi(M)$ is well defined. Moreover, from the set of generators we can explicitly construct the group $\Gamma_3^2(2,4)/\Gamma_3(2,4)$ and compute~\eqref{Phi} in order to see if there are choices of the matrix $M$ such that $\Phi(M)$ does not vanish identically.

A direct computation in \emph{Mathematica} shows that there are only 42 (from the 378 we started with) choices of the matrix $M$ such that $\Phi(M)(\tau)$ does not vanish identically, exactly the ones such that if $M=\sm{m_1'&m_2'\\}{m_1''&m_2''}$ then $m_1'=m_2'$. For instance, take 
\[M=\left(\begin{smallmatrix}
0&0\\
0&0\\
1&1\\
0&0\\
0&1\\
1&1
\end{smallmatrix}\right),\]
then \[\Phi(M)(\tau)=16\sum_{i=1}^4W(N_i)(\tau),\]
where $N_1=M$ and
\begin{align*}
N_2=\left(\begin{smallmatrix}
0&0\\
0&0\\
1&1\\
0&0\\
1&0\\
1&1
\end{smallmatrix}\right),\;
N_3=\left(\begin{smallmatrix}
0&0\\
0&0\\
1&1\\
1&1\\
0&1\\
1&1
\end{smallmatrix}\right),\;
N_4=\left(\begin{smallmatrix}
0&0\\
0&0\\
1&1\\
1&1\\
1&0\\
1&1
\end{smallmatrix}\right).
\end{align*} 
\endproof

In fact, it is not even possible to construct a map by ``squaring coordinates'' as in the genus two case. We will show that the coordinate ring of $\overline{\HH_3/\Gamma_3^2(2,4)}$ is not generated by squares of elements of the coordinate ring of $\overline{\HH_3/\Gamma_3(2,4)}$. 

The main difference from the genus two case is that in genus three there is a non-trivial algebraic relation between second order theta constants. Indeed by~\cite{runge1} we know that
\[A(\Gamma_3(2,4),v_\Theta)=\CC[f_a]/(R_{16}),\]
where 
\begin{equation*}
R_{16}=P_8(f_{000}^2,\dots,f_{111}^2)+q\cdot Q_4(f_{000}^2,\dots,f_{111}^2),
\end{equation*}
with $P_8$ and $Q_4$ polynomials in the $f_a^2$ of degree 8 and 4 respectively and $q=\prod_{a\in\FF_2^3}f_a$.
Its expression is simpler in terms of theta constants, namely
\[R_{16}=2^3\sum_{m\text{ even}}\vartheta_m^{16}(\tau)-\Big(\sum_{m\text{ even}}\vartheta_m^8(\tau)\Big)^2.\]
One can move from one expression to the other by means of the identities~\eqref{riemann relation} and \eqref{riemann relation m1}, recovering in this way the explicit expression of the polynomials $P_8$ and $Q_4$.

It is easily checked that $q\in A(\Gamma_3^2(2,4),v_\Theta^2)$ so this ring contains
\[R:=\CC[f_a^2,q]/(P_8+q\cdot Q_4,q^2=\textstyle{\prod_a}f_a^2).\]
By Serre's criterion~\cite{G} it is easy to show that $R$ is normal. Since $R$ is a complete intersection ring it is Cohen-Macaulay, so all that remains is to verify that it is regular in codimension 1, which we do by computer using \emph{Macaulay2}~\cite{m2}.
So we conclude that $A(\Gamma_3^2(2,4),v_\Theta^2)$ is not generated by squares of elements of $A(\Gamma_3(2,4),v_\Theta^2)$.

In genera higher than three many algebraic relations appear between second order theta constants, so a possible interpretation of the map that squares the coordinates of a projective space of suitable dimension would need a deeper analysis. We can see how the genus two case is peculiar from this point of view, since the principal results of this paper cannot be generalized directly to higher genera.

\end{document}